%% file: document.tex
\documentclass[11pt,a4paper]{article}
\usepackage[bib=library]{mathesis}

\date{\today}

\author{
	{\large Raúl Tempone,  Sören Wolfers\footnote{Corresponding author. Email address: \texttt{soeren.wolfers@kaust.edu.sa}}}\\
	{\small Computer, Electrical and Mathematical Sciences \& Engineering}\\
	{\small  King Abdullah University of Science and Technology (KAUST)}\\
}

\title{Smolyak's algorithm: A powerful black box for the acceleration of scientific computations}

\renewcommand{\i}{k} 
\newcommand{\mi}{\bm{\i}} 
\newcommand{\cd}{n} 
\newcommand{\vsc}{Y} 
\newcommand{\nm}{\mathcal{A}} 
\newcommand{\w}{\gamma}
\newcommand{\wmi}{\bm{\w}}
\renewcommand{\c}{\beta}
\newcommand{\cmi}{\bm{\c}}
\newcommand{\work}{\text{Work}}
\newcommand{\smol}{\mathcal{S}}

\DeclareMathOperator{\sumoperator}{\Sigma}
\newcommand{\dsum}[1]{{\sumoperator}_{#1}}
\newcommand{\ddiff}[1]{\Delta_{#1}}
\newcommand{\mix}{\text{mix}}
\newcommand{\mis}{\mathcal{I}}
\newcommand{\ml}{\mathcal{M}}

\newcommand{\boundf}{e}
\renewcommand{\O}{\mathcal{O}}
\newcommand{\vsb}{\mathcal{E}}
\newcommand{\vsbg}{\vsb_{(\boundf_j)_{j=1}^{\cd}}}
\newcommand{\workf}{w}

\newcommand{\vsunid}{H^{\beta}([0,1])}
\newcommand{\vsunic}{H^{\alpha}([0,1])}

 \newcommand{\comp}[1]{#1^{c}}
\newcommand{\lc}{s}
\newcommand{\lw}{t}
\newcommand{\lwmi}{\bm{\lw}}
\newcommand{\lcmi}{\bm{\lc}}
\newcommand{\xmi}{\bm{x}}
\newcommand{\ymi}{\bm{y}}
\newcommand{\mumi}{\bm{\mu}}
\newcommand{\numi}{\bm{\nu}}
\newcommand{\compseq}{\N^{\infty}_{c}}
\newcommand{\domPS}{\Gamma}
\newcommand{\ps}{y}
\newcommand{\pde}{u}

\newcommand{\rs}{f}
\newcommand{\psmi}{\bm{\ps}}
\newcommand{\ict}{l}
\newcommand{\mict}{\bm{\ict}}
\newcommand{\lin}{\mathcal{L}}

\begin{document}
\maketitle
\abstract{We provide a general discussion of Smolyak's algorithm for the acceleration of scientific computations. The algorithm first appeared in Smolyak's work on multidimensional integration and interpolation. Since then, it has been generalized in multiple directions and has been associated with the keywords: sparse grids, hyperbolic cross approximation, combination technique, and multilevel methods. Variants of Smolyak's algorithm have been employed in the computation of high-dimensional integrals in finance, chemistry, and physics, in the numerical solution of partial and stochastic differential equations, and in uncertainty quantification. Motivated by this broad and ever-increasing range of applications, we describe a general framework that summarizes fundamental results and assumptions in a concise application-independent manner.
\newline

 \textbf{Keywords} Smolyak algorithm, sparse grids, hyperbolic cross approximation, combination technique, multilevel methods
 }
 
\pagenumbering{roman}
\pagestyle{headings}
\pagenumbering{arabic}

\input{./content/main}

\input{./content/applications}
\appendix
\section{Exponential sums}
\input{./content/exponentialsums}

\clearpage

\end{document}

%% file: content/main.tex
\section{Introduction}
\label{sec:intro}
We study Smolyak's algorithm for the convergence acceleration of general numerical approximation methods
\begin{equation*}
\begin{split}
\nm\colon \N^\cd:=\{0,1,\dots,\}^{\cd}&\to \vsc,
\end{split}
\end{equation*}
which map discretization parameters $\mi=(\i_1,\dots,\i_\cd)\in\N^\cd$ to outputs $\nm(\mi)$ in a Banach space $\vsc$. 

For instance, a straightforward way to approximate the integral of a function $f\colon [0,1]^n\to\R$ is to employ tensor-type quadrature formulas, which evaluate $f$ at the nodes of a regular grid within $[0,1]^n$. This gives rise to an approximation method where $\i_j$ determines the grid resolution in direction of the $j$-th coordinate axis, $j\in\{1,\dots,\cd\}$. Smolyak himself derived and studied his algorithm in this setting, where it leads to evaluations in the nodes of \emph{sparse grids} \cite{smolyak1963quadrature, Zenger91}.  Another example, which emphasizes the distinctness of sparse grids and the general version of Smolyak's algorithm considered in this work, is integration of a univariate function $f\colon \R\to\R$ that is not compactly supported but exhibits sufficient decay at infinity. In this case, $\i_1$ could as before determine the resolution of regularly spaced quadrature nodes and $\i_2$ could be used to determine a truncated quadrature domain. Smolyak's algorithm  then leads to quadrature nodes whose density is high near the origin and decreases at infinity, as intuition would dictate.

 
 To motivate Smolyak's algorithm, assume that the approximation method $\nm$ converges to a limit $\nm_{\infty}\in \vsc$ at the rate
\begin{equation}
\label{eq:com:example}
\norm{\nm(\mi)-\nm_\infty}{\vsc}\leq K_1 \sum_{j=1}^{\cd}\i_j^{-\c_j} \quad\forall \mi\in\N^{\cd}
\end{equation}
and requires the work 
\begin{equation}
\label{eq:com:work}
\work(\nm(\mi))= K_2\prod_{j=1}^{\cd}\i_j^{\w_j} \quad\forall \mi\in\N^{\cd}
\end{equation}
for some $K_1>0,K_2>0$ and $\c_j>0$, $\w_j>0$, $j\in\{1,\dots,\cd\}$. 
An approximation of $\nm_{\infty}$ with accuracy $\epsilon>0$ can then be obtained with the choice
\begin{equation}
\label{eq:com:straightforwarddef}
\i_j:=-\Big(\frac{\epsilon}{\cd K_1}\Big)^{-1/\c_j},\quad j\in\{1,\dots,\cd\},
\end{equation}
which requires the work
\begin{equation}
\label{eq:com:curse}
C(\cd,K_1,K_2,\w_1,\dots,\w_\cd,\c_1,\dots,,\c_\cd)\epsilon^{-(\w_1/\c_1+\dots+\w_\cd/\c_\cd)}.
\end{equation}
Here and in the remainder of this work we denote by $C(\dots)$ generic constants that depend only on the quantities in parentheses but may change their value from line to line and from equation to equation.

The appearance of the sum $\w_1/\c_1+\dots+\w_\cd/\c_\cd$ in the exponent above is commonly referred to as the \emph{curse of dimensionality}. Among other things, we will show (see \Cref{exa}) that if the bound in \Cref{eq:com:example} holds in a slightly stronger sense, then Smolyak's algorithm can replace this dreaded sum by $\max_{j=1}^{\cd}\w_j/\c_j$, which means that it yields convergence rates that are, up to possible logarithmic factors, independent of the number of discretization parameters. In the general form presented here, Smolyak's algorithm forms linear combinations of the values $\nm(\mi)$, $\mi\in\N^{\cd}$, based on
\begin{enumerate}
\item an infinite decomposition of $\nm_{\infty}$ and
\item a knapsack approach to truncate this decomposition.
\end{enumerate} 
Since the decomposition is independent of the particular choice of $\nm$ and the truncation relies on easily verifiable assumptions on the decay and work of the decomposition terms, Smolyak's algorithm is a powerful black box for the non-intrusive acceleration of scientific computations. In the roughly 50 years since its first description, applications in various fields of scientific computation have been described; see, for example, the extensive survey article \cite{bungartz2004sparse}. The goal of this work is to summarize previous results in a common framework and thereby encourage further research and exploration of novel applications. 
While some of the material presented here may be folklore knowledge in the sparse grids community, we are not aware of any published sources that present this material in a generally applicable fashion.
\newline

The remainder of this work is structured as follows. In \Cref{sec:decomp}, we introduce the infinite decomposition of $\nm_{\infty}$ that is at the core of Smolyak's algorithm. In \Cref{sec:com:truncation}, we introduce spaces of approximation methods $\nm\colon \N^\cd\to\vsc$ that allow for efficient solutions of the resulting truncation problem. In \Cref{sec:com:convergence}, we derive explicit convergence rates for Smolyak's algorithm in common examples of such spaces. Finally, in \Cref{sec:com:applications}, we discuss how various previous results can be deduced within the framework presented here.

\section{Decomposition}
\label{sec:decomp} Smolyak's algorithm is based on a decomposition of $\nm_{\infty}$ that is maybe most simply presented in the continuous setting. Here, Fubini's theorem and the fundamental theorem of calculus show that any function $f\colon \Rnn^{\cd}:=[0,\infty)^{\cd}\to \vsc$ with $f\equiv 0$ on $\partial\Rnn^{\cd}$ satisfies
\begin{equation}
\label{eq:com:rectint}
f(x)=\int_{\prod_{j=1}^\cd [0,x_j]}\partial_{1}\dots\partial_{\cd}f(s) \;ds\quad\forall x\in \Rnn^{\cd},
\end{equation}
questions of integrability and differentiability aside. Moreover, if $f$ converges to a limit $f_{\infty}\in\vsc$ as $\min_{j=1}^{\cd}x_j\to\infty$, then 
\begin{equation}
\label{eq:com:cdecomp}
f_{\infty}=\lim_{\min_{j=1}^{\cd}x_j\to\infty}\int_{\prod_{j=1}^\cd [0,x_j]}\partial_{1}\dots\partial_{\cd}f(s) \;ds=\int_{\Rnn^{\cd}}\partial_{\mix}f(s)\,ds,
\end{equation}
where we introduced the shorthand $\partial_{\mix}$ for the mixed derivative $\partial_{1}\dots\partial_{\cd}$.
The crucial observation is now that an approximation of $f_{\infty}$ can be achieved not only by rectangular truncation of the integral in \Cref{eq:com:cdecomp}, which according to \Cref{eq:com:rectint} is equivalent to a simple evaluation of $f$ at a single point, but also by truncation to more complicated domains. These domains should ideally correspond to large values of $\partial_{\mix}f$ in order to minimize the truncation error, but also have to take into consideration the associated computational work. 

To transfer the decomposition in \Cref{eq:com:cdecomp} to the discrete setting, we denote by $\vsc^{\N^\cd}:=\{\nm\colon\N^\cd\to\vsc\}$ the space of all functions from $\N^\cd$ into the Banach space $\vsc$. Next, we define the discrete unidirectional difference and sum operators
\begin{equation*}
\begin{split}
&\hspace{3em}\ddiff{j}\colon  \vsc^{\N^{\cd}}\to\vsc^{\N^{\cd}}\\
(\ddiff{j}\nm)(\mi):=&\begin{cases}\nm(\i_1,\dots,\i_{\cd})-\nm(\i_1,\dots,\i_{j-1},\i_{j}-1,\i_{j+1},\dots,\i_{\cd})\quad\text{if }\i_j>0,\\
\nm(\i_1,\dots,\i_{\cd})\quad \text{else},
\end{cases}
\end{split}
\end{equation*}
\begin{equation*}
\begin{split}
&\dsum{j}:=\ddiff{j}^{-1}\colon \vsc^{\N^{\cd}}\to \vsc^{\N^{\cd}}\\
(\dsum{j}&\nm)(\mi):=\sum_{s=0}^{\i_j}\nm(\i_1,\dots,\i_{j-1},s,\i_{j+1},\dots,\i_\cd),
\end{split}
\end{equation*}
Finally, we 
introduce their compositions, the mixed difference operator
\begin{equation*}
\ddiff{\mix}:=\ddiff{1}\circ\dots\circ\ddiff{\cd}\colon \vsc^{\N^{\cd}}\to\vsc^{\N^{\cd}},
\end{equation*}
and the rectangular sum operator
\begin{equation*}
	\dsum{R}:=\dsum{1}\circ\dots\circ\dsum{\cd}\colon \vsc^{\N^{\cd}}\to\vsc^{\N^{\cd}},
\end{equation*}
which replace the mixed derivative and integral operators that map $f\colon\R^{\cd}\to\vsc$ to $f\mapsto \partial_{\mix}f$ and $x\mapsto \int_{\prod_{j=1}^{\cd}[0,x_j]} f(s)\,ds$, respectively. 

The discrete analogue of \Cref{eq:com:rectint} is now a matter of simple algebra.
\begin{proposition}
\label{pro:com:1}
\begin{enumerate}[(i)]
\item We have $\dsum{R}=\ddiff{\mix}^{-1}$, that is

$$
\nm(\mi)=\sum_{s_1=0}^{\i_1}\dots\sum_{s_\cd=0}^{\i_{\cd}} \ddiff{\mix}\nm(s_1,\dots,s_{\cd})\quad\forall \mi\in\N^{\cd}.
$$
\item We have $\ddiff{\mix}=\sum_{\mathbf{e}\in \{0,1\}^{\cd}} (-1)^{|\mathbf{e}|_1}S_{\mathbf{e}}$, where $S_{\mathbf{e}}$ is the shift operator defined by 
\begin{equation*}
(S_{\mathbf{e}}\nm)(\mi):=\begin{cases}\nm(\mi-\mathbf{e}),\quad\text{if }\mi-\mathbf{e}\in\N^{\cd}\\
0\quad\text{else}\end{cases}.
\end{equation*} 
\end{enumerate}
\end{proposition}
\begin{proof}
Part (i) follows directly from the commutativity of the operators $\{\dsum{j}\}_{j=1}^{\cd}$.
Part (ii) follows from plugging the representation $\ddiff{j}=\Id-S_{\mathbf{e}_j}$, where $\mathbf{e}_j$ is the $j$-th standard basis vector in $\N^{\cd}$, into the definition $\ddiff{\mix}=\ddiff{1}\circ\dots\circ\ddiff{\cd}$, and subsequent expansion.
\end{proof}
Part (i) of the previous proposition shows that, ignoring questions of convergence, discrete functions $\nm\colon \N^{\cd}\to \vsc$ with limit $\nm_{\infty}$ satisfy
\begin{equation}
\label{eq:com:ddecomp}
\nm_{\infty}=\sum_{\mi\in\N^{\cd}}\ddiff{\mix}\nm(\mi)
\end{equation}
in analogy to \Cref{eq:com:cdecomp}. In the next section, we define spaces of discrete functions for which this sum converges absolutely and can be efficiently truncated. We conclude this section by the observation that a necessary condition for the sum in \Cref{eq:com:ddecomp} to converge absolutely is that the unidirectional limits $\nm(\i_1,\dots,\infty,\dots,\i_{\cd}):=\lim_{\i_j\to\infty}\nm(\i_1,\dots,\i_j,\dots,\i_\cd)$ exist. Indeed, by part (i) of the previous proposition, these limits correspond to summation of $\ddiff{\mix}\nm$ over hyperrectangles that are growing in direction of the $j$-th coordinate axis and fixed in all other directions. For instance, in the context of time-dependent partial differential equations this implies stability requirements for the underlying numerical solver, prohibiting explicit time-stepping schemes that diverge when the space-discretization is refined while the time-discretization is fixed.

\section{Truncation}
\label{sec:com:truncation}
For any index set $\mis\subset\N^{\cd}$, we may define Smolyak's algorithm as the approximation of $\nm_{\infty}$ that is obtained by truncation of the infinite decomposition in \Cref{eq:com:ddecomp} to $\mis$,
\begin{equation}
\smol_{\mis}(\nm):=\sum_{\mi\in \mis}\ddiff{\mix}\nm(\mi).
\end{equation}
By definition of $\ddiff{\mix}\nm$, the approximation $\smol_{\mis}(\nm)$ is a linear combination of the values $\nm(\mi)$, $\mi\in\N^{\cd}$ (see \Cref{ssec:com:com} for explicit coefficients). This is the reason for the name \emph{combination technique} that was given to approximations of this form in the context of the numerical approximation of partial differential equations \cite{GriebelSchneiderZenger1992}. 
When one talks about \emph{the} Smolyak algorithm, or \emph{the} combination technique, a particular truncation is usually implied. The general idea here is to include those indices for which the ratio between contribution (measured in the norm of $\vsc$) and required work of the corresponding decomposition term is large. 
To formalize this idea, we require decay of the norms of the decomposition terms and bounds on the work required for their evaluation. 
To express the former, we define for strictly decreasing functions $\boundf_j\colon\N\to\Rp:=(0,\infty)$, $j\in\{1,\dots,\cd\}$ the spaces
\begin{equation*}
\vsbg(\vsc):=\big\{\nm\colon \N^{\cd}\to\vsc :\exists K_1>0 \;\forall \mi\in\N^{\cd}\; \norm{\ddiff{\mix}\nm(\mi)}{\vsc}\leq K_1\prod_{j=1}^{\cd}\boundf_j(\i_j)\big\}.
\end{equation*} 

\begin{proposition}
	\label{pro:com:2}
\begin{enumerate}[(i)]
	\item If
	\begin{equation}
	\label{eq:com:summability}
	\sum_{\mi\in\N^{\cd}}\prod_{j=1}^{\cd}\boundf_j(\i_j)<\infty,
	\end{equation} 
	then any $\nm\in\vsbg(\vsc)$ has a limit $\nm_{\infty}:=\lim_{\min_{j=1}^{\cd}\i_j\to\infty}\nm(\mi)$. Furthermore, the decomposition in \Cref{eq:com:ddecomp} holds and converges absolutely.
\item The spaces $\vsbg(\vsc)$ are linear subspaces of $Y^{\N^{\cd}}$.

\item (\emph{Error expansions}) Assume that the ratios $\boundf_j(k)/\boundf_j(k+1)$ are uniformly bounded above for $k\in\N$ and $j\in\{1,\dots,\cd\}$. For $\mi\in\N^{\cd}$ and $J\subset\{1,\dots,\cd\}$ let  $\mi_J:=(\i_j)_{j\in J}\in \N^{|J|}$. If the approximation error can be written as 
\begin{equation*}
\nm(\mi)-\nm_{\infty}=\sum_{\varnothing\not=J\subset\{1,\dots,\cd\}}\nm_J(\mi_{J})\quad\forall \mi\in\N^{\cd}
\end{equation*}
with functions $\nm_{J}\colon\N^{|J|}\to\vsc$, $J\subset\{1,\dots,\cd\}$ that satisfy 
 \begin{equation*}
 \norm{\nm_J(\mi_J)}{\vsc}\leq \prod_{j\in J}\boundf_j(\i_j)
 \end{equation*}
then 
$$
\nm\in\vsbg(\vsc).
$$

\item (\emph{Multilinearity} \cite{WolfersSparse}) Assume $(\vsc_i)_{i=1}^m$ and $Y$ are Banach spaces and $\ml\colon \prod_{i=1}^m\vsc_i\to \vsc$ is a continuous multilinear map. If
\begin{equation*}
\nm_i\in\vsb_{(e_{j})_{j=\cd_1+\dots+\cd_{i-1}+1}^{\cd_1+\dots+\cd_i}}(\vsc_i)\quad \forall\,i\in\{1,\dots,m\},
\end{equation*}
then
\begin{equation*}
\ml(\nm_1,\dots,\nm_m)\in\vsb_{(e_{j})_{j=1}^{\cd}}(\vsc),
\end{equation*}
 where $\cd:=\cd_1+\dots+\cd_{m}$ and
\begin{equation*}
\ml(\nm_1,\dots,\nm_m)(\mi):=\ml(\nm_1(\mi_1),\dots,\nm_m(\mi_m))\quad\forall  \mi=(\mi_1,\dots,\mi_m)\in\N^{\cd}.
\end{equation*}
\end{enumerate}
\end{proposition}
\begin{proof}
Since $\vsc$ is a Banach space, the assumption in part (i) shows that for any $\nm\in\vsbg(\vsc)$ the infinite sum in \Cref{eq:com:ddecomp} converges absolutely to some limit $\bar{\nm}$. Since rectangular truncations of this sum yield point values $\nm(\mi)$, $\mi\in\N^{\cd}$ by part (i) of 
\Cref{pro:com:1}, the limit $\nm_{\infty}:=\lim_{\min_{j=1}^{\cd}\i_j\to\infty}\nm(\mi)$ exists and equals $\bar{\nm}$.	
Part (ii) follows from the triangle inequality. 

For part (iii), observe that by part (ii) it suffices to show $\nm_{J}\in\vsbg(\vsc)$ for all $J\subset\{1,\dots,\cd\}$, where we consider $\nm_{J}$ as functions on $\N^{\cd}$ depending only on the parameters indexed by $J$. Since $\ddiff{\mix}=\ddiff{\mix}^{J}\circ\ddiff{\mix}^{J^C}$, where $\ddiff{\mix}^{J}$ denotes the mixed difference operator acting on the parameters in $J$, we then obtain
\begin{equation*}
\ddiff{\mix}\nm_J(\mi)=\begin{cases}
\ddiff{\mix}^{J}\nm_{J}(\mi_J) \quad\text{if }\;\forall j\in J^{C}:\mi_{j}=0\\
0\quad\text{else}.
\end{cases}
\end{equation*}
Hence, it suffices to consider $J=\{1,\dots,\cd\}$. In this case, the assumption $\norm{\nm_{J}(\mi_{J})}{\vsc}\leq C\prod_{j\in J}\boundf_{j}(\i_j)$ is equivalent to $\ddiff{\mix}^{-1}\nm_{J}\in \vsbg(\vsc)$. Thus, it remains to show that $\ddiff{\mix}$ preserves $\vsbg(\vsc)$. This holds by part (ii) of this proposition together with part (ii) of \Cref{pro:com:1} and the fact that shift operators preserve $\vsbg(\vsc)$, which itself follows from the assumption that the functions $\boundf_j(\cdot)/\boundf_j(\cdot+1)$ are uniformly bounded.

Finally, for part (iv) observe that by multilinearity of $\ml$ we have
\begin{equation*}
\ddiff{\mix}\ml(\nm_1,\dots,\nm_m)=\ml(\ddiff{\mix}^{(1)}\nm_1,\dots,\ddiff{\mix}^{(m)}\nm_m),
\end{equation*}
where the mixed difference operator on the left hand side acts on $\cd=\cd_1+\dots+\cd_m$ coordinates, whereas those on the right hand side only act on the $\cd_i$ coordinates of $\nm_i$. By continuity of $\ml$ we have
\begin{equation*}
\norm{\ml(\ddiff{\mix}^{(1)}\nm_1,\dots,\ddiff{\mix}^{(m)}\nm_m)(\mi)}{\vsc}\leq C\prod_{i=1}^{m}\norm{\ddiff{\mix}^{(i)}\nm_i(\mi_i)}{\vsc_i},
\end{equation*}
for some $C>0$, from which the claim follows.
\end{proof}
Parts (iii) and (iv) of the previous proposition provide sufficient conditions to verify $\nm\in\vsbg(\vsc)$ without analyzing mixed differences directly.
\begin{example}
\label{exa}
	\begin{enumerate}[(i)]
		\item 
		After an exponential reparametrization, the assumptions in  \Cref{eq:com:example,eq:com:work} become
	\begin{equation*}
	\norm{\nm(\mi)-\nm_\infty}{\vsc}\leq K_1 \sum_{j=1}^{\cd}\exp(-\c_j\i_j) \quad\forall \mi\in\N^{\cd}
	\end{equation*}
and
	\begin{equation*}
	\work(\nm(\mi))= K_2\prod_{j=1}^{\cd}\exp(\w_j\i_j) \quad\forall \mi\in\N^{\cd},
	\end{equation*} respectively.
		  If we slightly strengthen the first and assume that 
		\begin{equation*}
		\nm(\mi)-\nm_{\infty}=\sum_{j=1}^{\cd}\nm_j(\i_j)\quad\forall \mi\in\N^\cd
		\end{equation*}
		with functions $\nm_j$ that satisfy 
		\begin{equation*}
		\norm{\nm_j(\i_j)}{\vsc}\leq C\exp(-\c_j\i_j),\quad\forall \i_j\in\N
		\end{equation*} for some $C>0$ and $\c_j>0$, $j\in\{1,\dots,\cd\}$, then 
		\begin{equation*}
		\nm\in\vsbg \quad\text{with } \boundf_{j}(\i_j):=\exp(-\c_j\i_j),
		\end{equation*}
		by part (iii) of \Cref{pro:com:2}. 
		\Cref{thm:main} below then shows that Smolyak's algorithm applied to $\nm$ requires only the work $\epsilon^{-\max_{j=1}^{\cd}\{\w_j/\c_j\}}$, up to possible logarithmic factors, to achieve the accuracy $\epsilon>0$.  
		
		\item Assume we want to approximate the integral of a function $\rs\colon  [0,1]\to\R$ but are only able to evaluate approximations $\rs_{\i_2}$, $\i_2\in\N$ of $\rs$ with increasing cost as $\i_2\to\infty$. Given a sequence $S_{\i_1}$, $\i_1\in\N$ of linear quadrature formulas, the straightforward approach would be to fix sufficiently large values of $\i_1$ and $\i_2$ and then approximate the integral of $\rs_{\i_2}$ with the quadrature formula $S_{\i_1}$. Formally, this can be written as
		\begin{equation*}
		\nm(\i_1,\i_2):=S_{\i_1}\rs_{\i_2}.
		\end{equation*}
		To show decay of the mixed differences $\ddiff{\mix}\nm$, observe that the application of quadrature formulas to functions is linear in both arguments, which means that we may write
		\begin{equation*}
		\nm(\i_1,\i_2)=\ml(S_{\i_1},\rs_{\i_2})=\ml(\nm_1(\i_1),\nm_2(\i_2))
		\end{equation*}
		where $\nm_1(\i_1):=S_{\i_1}$, $\nm_2(\i_2):=\rs_{\i_2}$, and $\ml$ is the application of linear functionals to functions on $[0,1]$. Assume, for example, that the functions $\rs_{\i_2}$ converge to $\rs$ in some Banach space $B$ of functions on $[0,1]$ as $\i_2\to\infty$, and that the quadrature formulas $S_{\i_1}$ converge to the integral operator $\int$ in the continuous dual space $B^*$ as $\i_1\to\infty$. The decay of the mixed differences $\ddiff{\mix}\nm(\i_1,\i_2)$ then follows from part (iv) of \Cref{pro:com:2}, since $\ml$ is a continuous bilinear map from $B^* \times B$ to $\R$. We will see in \Cref{ssec:mlquadrature} below that the application of Smolyak's algorithm in this example yields so called \emph{multilevel quadrature formulas}. This connection between Smolyak's algorithm and multilevel formulas was observed in \cite{harbrecht2013multilevel}.
		
		\item \label{exa3} Assume that we are given approximation methods $\nm_j\colon\N\to \vsc_j$, $j\in\{1,\dots\cd\}$ that converge at the rates $\norm{\nm_{j}(\i_j)-\nm_{\infty,j}}{\vsc_j}\leq \boundf_j(\i_j)$ to limits $\nm_{\infty,j}\in \vsc_j$, where $\boundf_j\colon\N\to \Rp$ are strictly decreasing functions. Define the tensor product algorithm
		
		\begin{equation*}
		\begin{split}
		\nm\colon\N^{\cd}\to \vsc:=\vsc_1\otimes\dots\otimes\vsc_{\cd},\\
		\nm(\mi):=\nm_{1}(\i_1)\otimes\dots\otimes\nm_{\cd}(\i_\cd).
		\end{split}
		\end{equation*}
		If the algebraic tensor product $\vsc$ is equipped with a norm that satisfies $\norm{y_1\otimes\dots\otimes y_\cd}{\vsc}\leq\norm{y_1}{\vsc_1}\dots \norm{y_{\cd}}{\vsc_{\cd}}$, then $\nm\in\vsbg(\vsc)$. Indeed, $\nm_j\in\vsb_{\boundf_j}(\vsc_j)$ by part (iii) of \Cref{pro:com:2}, thus $\nm\in\vsbg(\vsc)$ by part (iv) of the same proposition. 

\end{enumerate}
\end{example}
Similar to the product type decay assumption on the norms $\norm{\ddiff{\mix}\nm(\mi)}{\vsc}$, which we expressed in the spaces $\vsbg$, we assume in the remainder that 
\begin{equation}
\label{eq:com:workfinal}
\work(\ddiff{\mix}\nm(\mi))\leq K_2\prod_{j=1}^{\cd}\workf_j(\i_j)\quad\forall \mi \in \N^{\cd}
\end{equation}
for some $K_2>0$ and increasing functions $\workf_j\colon\N\to \Rp$. By part (ii) of \Cref{pro:com:1}, such a bound follows from the same bound on the evaluations $\nm(\mi)$ themselves.

\subsection{Knapsack problem}
The goal of this subsection is to  describe quasi-optimal truncations of the decomposition in \Cref{eq:com:ddecomp} for functions $\nm\in\vsbg(\vsc)$ that satisfy \Cref{eq:com:workfinal}. 
Given a work budget $W>0$, a quasi-optimal index set solves the knapsack problem
\begin{equation}
\label{eq:com:knapsack}
\begin{split}
\max_{\mis\subset\N^{\cd}}&\quad|\mis|_{\boundf}:=K_1\sum_{\mi\in\mis}\prod_{j=1}^{\cd}\boundf_{j}(\i_j)\\
\text{subject to}&\quad|\mis|_{\workf}:=K_2\sum_{\mi\in\mis}\prod_{j=1}^{\cd}\workf_j(\i_j)\leq W.
\end{split}
\end{equation}
The term that is maximized here is motivated by 
\begin{equation*}
\norm{\smol_{\mis}(\nm)-\nm_{\infty}}{\vsc}=\norm{\sum_{\mi\in\comp{\mis}}\ddiff{\mix}\nm(\mi)}{\vsc}\approx\sum_{\mi\in\comp{\mis}}\norm{\ddiff{\mix}\nm(\mi)}{\vsc}\approx |\comp{\mis}|_{\boundf}
\end{equation*}

\Cref{pro:com:knapsack} below shows that for any $W>0$ the knapsack problem has an optimal value. However, finding corresponding optimal sets is NP-hard \cite[Section 1.3]{martello1990knapsack}. As a practical alternative one can use Dantzig's approximation algorithm \cite[Section 2.2.1]{martello1990knapsack}, which selects indices for which the ratio between contribution and work is above some threshold $\delta(W)>0$,
\begin{equation}
\label{eq:com:dantzig}
\mis_{W}:=\{\mi\in\N^{\cd}:\prod_{j=1}^{\cd}\boundf_{j}(\i_j)/\workf_{j}(\i_j)>\delta(W)\},
\end{equation}
where $\delta(W)$ is chosen minimally such that $|\mis_{W}|_{\workf}\leq W$.
\begin{proposition}
	\label{pro:com:knapsack}
	\begin{enumerate}[(i)]
		\item The knapsack problem in \Cref{eq:com:knapsack} has a (not necessarily unique) solution, in the sense that a maximal value of $|\mis|_{\boundf}$ is attained. We denote this maximal value by $E^*(W)$.
		\item Any set $\mis^{*}$ for which $|\mis|_{\boundf}=E^*(W)$ is finite and downward closed: If $\mi\in\mis^{*}$ and  $\tilde{\mi}\in\N^{\cd}$ satisfies $\tilde{\mi}\leq \mi$ componentwise, then $\tilde{\mi}\in\mis^{*}$. The same holds for the set $\mis_{W}$ from \Cref{eq:com:dantzig}.
		
		\item The set $\mis_{W}$ from \Cref{eq:com:dantzig} satisfies
		\begin{equation*}
		|\mis_{W}|_{\boundf}\geq \frac{|\mis_{W}|_{\workf}}{W} E^*(W).
		\end{equation*}
		This means that if $\mis_{W}$ uses all of the available work budget, $|\mis_{W}|_{\workf}=W$, then it is a solution to the knapsack problem. In particular, Dantzig's solutions are optimal for the work $|\mis_{W}|_{\workf}$ they require, but not necessarily for the work  $W$ they were designed for.
	\end{enumerate}
\end{proposition}
\begin{proof}
There is an upper bound $N$ on the cardinality of admissible sets in \Cref{eq:com:knapsack} since the functions $\workf_j$ are increasing and strictly positive. Furthermore, replacing an element $\mi$ of an admissible set by $\tilde{\mi}$ with $\tilde{\mi}\leq\mi$ decreases $|\cdot|_{\workf}$ and increases $|\cdot|_{\boundf}$. This proves parts (i) and (ii), as there are only finitely many downward closed sets of cardinality less than $N$ (for example, all such sets are subsets of $\{0,\dots,N-1\}^{\cd}$). Part (iii) follows directly from the inequality $|\mis_{W}|_\boundf/|\mis_{W}|_{\workf}\geq |\mis^*|_{\boundf}/|\mis^*|_{\workf}$, where $\mis^{*}$ is a set that attains the maximal value $E^*(W)$. 
\end{proof}
Even in cases where no bounding functions $\boundf_j$ and $\workf_j$ are available, parts (ii) and (iii) of the previous proposition serve as motivation for adaptive algorithms that progressively build a downward closed set $\mis$ by adding at each step a multi-index that maximizes a gain-to-work estimate \cite{MR2163199,MR2366325}.

\subsection{Combination rule}
\label{ssec:com:com}
Part (ii) of \Cref{pro:com:1} provides a way to express the approximations $\smol_{\mis}(\nm)$ in a succinct way as linear combinations of different values of $\nm$. This yields the \emph{combination rule}, which in its general form says that
\begin{equation*}
\smol_{\mis}(\nm)=\sum_{\mi\in\mis}c_{\mi}\nm(\mi)
\end{equation*}
with
\begin{equation}
\label{eq:com:coeff}
c_{\mi}=\sum_{e\in\{0,1\}^{\cd}: \mi+e\in\mis}(-1)^{|e|_1}
\end{equation}
for any downward closed set $\mis$.
It is noteworthy that $c_{\mi}=0$ for all $\mi$ with $\mi+(1,\dots,1)\in\mis$, because for such $\mi$ the sum in \Cref{eq:com:coeff} is simply the expansion of $(1-1)^{\cd}$.

When $\mis$ is a standard simplex, $\mis=\{\mi\in\N^{\cd}:|\mi|_{1}\leq L\}$, the following explicit formula holds \cite{Wasilkowski1995}:
\begin{equation*}
c_{\mi}=\begin{cases}
(-1)^{L-|\mi|_{1}}\binom{\cd-1}{L-|\mi|_{1}}\quad \text{if } L-\cd+1\leq |\mi|_{1}\leq L\\
0\quad \text{else}.
\end{cases}
\end{equation*}

\section{Convergence analysis}
\label{sec:com:convergence}

\subsection{Finite-dimensional case}
\label{ssec:com:finite}
We consider an approximation method $\nm\in \vsbg(\vsc)$ with 
\begin{equation}
\label{eq:f1}
\boundf_j(\i_j)=K_{j,1}\exp(-\c_j\i_j)(\i_j+1)^{\lc_j}\quad\forall j\in\{1,\dots,\cd\}
\end{equation}
and assume that
\begin{equation}
\label{eq:f1}
\work(\nm(\mi))\leq \prod_{j=1}^{\cd}K_{j,2}\exp(\w_j\i_j)(\i_j+1)^{\lw_j} \quad\forall \mi \in \N^{\cd}
\end{equation}
with $K_{j,1}>0$, $K_{j,2}>0$, $\c_j> 0$, $\w_j> 0$, $\lc_j\geq 0$, $\lw_j\geq 0$. The required calculations with $\lcmi\equiv \lwmi\equiv 0$ were previously done in various specific contexts, see for example \cite{haji2015multi}.
According to \Cref{pro:com:knapsack}, quasi-optimal index sets are given by 
\begin{equation*}
\begin{split}
\mis_{\delta}:&=\big\{\mi\in\N^{\cd}:\prod_{j=1}^{\cd}\frac{K_{j,1}\exp(-\c_j\i_j)(\i_j+1)^{\lc_j}}{K_{j,2}\exp(\w_j\i_j)(\i_j+1)^{\lw_j}}>\delta \big\}\\
&=\big\{\mi\in\N^{\cd}:\frac{K_1}{K_2}\exp(-(\cmi+\wmi)\cdot\mi)\prod_{j=1}^{\cd}(\i_j+1)^{\lc_j-\lw_j}>\delta \big\}
\end{split}
\end{equation*}
for $\delta>0$, where $K_1:=\prod_{j=1}^{\cd}K_{j,1}$, $K_2:=\prod_{j=1}^{\cd}K_{j,2}$, and  $\cmi:=(\c_1,\dots,\c_{\cd})$, $\wmi:=(\w_1,\dots,\w_{\cd})$.
For the analysis in this section, we use the slightly simplified sets
\begin{equation*}
\mis_{L}:=\left\{\mi\in\N^{\cd}:\exp((\cmi+\wmi)\cdot\mi)\leq \exp(L) \right\}=\left\{\mi\in\N^{\cd}:(\cmi+\wmi)\cdot \mi\leq L \right\},
\end{equation*}
with $L\to\infty$, where, by abuse notation, we distinguish the two families of sets by the subscript letter.

The work required by $\smol_{L}(\nm):=\smol_{\mis_{L}}(\nm)$ satisfies
\begin{equation*}
\work(\smol_{L}(\nm))\leq \sum_{\mi\in\mis_{L}}\prod_{j=1}^{\cd}K_{j,2}\exp(\w_j\i_j)(\i_j+1)^{\lw_j}=K_2\sum_{(\cmi+\wmi)\cdot\mi\leq L}\exp(\wmi\cdot\mi)(\mi+\mathbf{1})^{\lwmi}
\end{equation*}
with $(\mi+\mathbf{1})^{\lwmi}:=\prod_{j=1}^{\cd}(\i_j+1)^{\lw_j}$.
Similarly, the approximation error satisfies
\begin{equation}
\label{eq:pt}
\norm{\smol_{L}(\nm)-\nm_{\infty}}{\vsc}\leq \sum_{\mi\in\comp{\mis_{L}}}\prod_{j=1}^{\cd}K_{j,1}\exp(-\c_j\i_j)\i_j^{\lc_j}=K_1 \sum_{(\cmi+\wmi)\cdot\mi>L}\exp(-\cmi\cdot\mi)(\mi+\mathbf{1})^{\lcmi}.
\end{equation}
The exponential sums appearing in the work and residual bounds above are estimated in the appendix of this work, with the results
\begin{equation}
\label{eq:com:expwork}
\work(\smol_{L}(\nm))\leq K_2C(\wmi,\lwmi,\cd)\exp(\frac{\rho}{1+\rho}L)(L+1)^{\cd^*-1+\lw^*}
\end{equation}
and
\begin{equation}
\label{eq:com:experror}
\norm{\smol_{L}(\nm)-\nm_{\infty}}{\vsc}\leq K_1C(\cmi,\lcmi,\cd)\exp(-\frac{1}{1+\rho}L)(L+1)^{\cd^*-1+\lc^*},
\end{equation}
where $\rho:=\max_{j=1}^{\cd}\w_j/\c_j$, $J:=\{j\in\{1,\dots,\cd\}:\w_j/\c_j=\rho\}$, $\cd^*:=|J|$, $\lc^*:=\sum_{j\in J}\lc_j$, $\lw^*:=\sum_{j\in J}\lw_j$. 
We may now formulate the main result of this section by rewriting the bound in \Cref{eq:com:expwork} in terms of the right-hand side of \Cref{eq:com:experror}.
\begin{theorem}
	\label{thm:main}
	 Under the previously stated assumptions on $\nm$ and for small enough $\epsilon>0$, we may choose $L>0$ such that 
\begin{equation*}
\norm{\smol_{L}(\nm)-\nm_{\infty}}{\vsc}\leq \epsilon
\end{equation*}
and
\begin{equation*}
\work(\smol_{L}(\nm))\leq K_1^{\rho}K_2C(\cmi,\wmi,\lcmi,\lwmi,\cd)\epsilon^{-\rho}|\log\epsilon|^{(\cd^*-1)(1+\rho)+\rho\lc^*+\lw^{*}}.
\end{equation*}
\hfill\qedsymbol
\end{theorem}
This means that we have eliminated the sum in the exponent of the bound in \Cref{eq:com:curse}, as announced in \Cref{sec:intro}. The additional logarithmic factors in \Cref{thm:main} vanish if the worst ratio of work and convergence exponents, $\rho$, is attained only for a single index $j_{\max}\in\{1,\dots,\cd\}$ and if $\lw_{j_{\max}}=\lc_{j_{\max}}=0$.
\begin{remark}
\label{rem:com:cov}
If $\wmi\equiv 0$ and $\cmi\equiv 0$, that is when both work and residual depend algebraically on all parameters, then an exponential reparametrization, $\exp(\tilde{\mi}):=\mi$, takes us back to the situation considered above. The preimage of $\mis_{L}=\{\tilde{\mi} : ({\lcmi}+{\lwmi})\cdot \tilde{\mi}\leq L\}$ under this reparametrization is $\{\mi : \prod_{j=1}^{\cd}\i_j^{\lc_j+\lw_j}\leq \exp(L)\}$, whence the name \emph{hyperbolic cross approximation} \cite{DuTeUl2015}.
\end{remark}
\begin{remark}
	\label{rem:orth}
When the terms $\ddiff{\mix}\nm(\mi)$, $\mi\in\N^{\cd}$ are orthogonal to each other, we may substitute the Pythagorean theorem for the triangle inequality in \Cref{eq:pt}. As a result, the exponent of the logarithmic factor in \Cref{thm:main} reduces to $(\cd^{*}-1)(1+\rho/2)+\rho \lc^*+\lw^*$.
\end{remark}

\subsection{Infinite-dimensional case}
\label{ssec:com:infinite}
The theory of the previous sections can be extended to the case $\cd=\infty$. In this case the decomposition in \Cref{eq:com:ddecomp} becomes
\begin{equation}
\label{eq:com:infdecomp}
\nm_{\infty}=\sum_{\mi\in\compseq}\ddiff{\mix}\nm(\mi),
\end{equation}
where $\compseq$ are the sequences with finite support, and $\ddiff{\mix}\nm(\mi)$ is defined as $\ddiff{1}\circ\dots\circ\ddiff{\cd_{\max}}\nm(\mi)$, where $\cd_{\max}$ is a bound on the support of $\mi$. 
In particular, every term in \Cref{eq:com:infdecomp} is a linear combination of values of $\nm$ with only finitely many nonzero discretization parameters.

We consider the case $\nm\in\vsbg(\vsc)$ for  
\begin{equation*}
\boundf_j(\i_j):=K_{j,1}\exp(-\beta_j\i_j)(\i_j+1)^{\lc} \quad \forall j\geq 1
\end{equation*}
 and $\lc\geq 0$, $K_1:=\prod_{j=1}^{\infty}K_{j,1}<\infty$ $\lc\geq 0$, and we assume constant computational work for the evaluation of the mixed differences $\ddiff{\mix}\nm(\mi)$, i.e. $\workf_j\equiv C$ in \Cref{eq:com:workfinal} for all $j\geq 1$. 
Similarly to the finite-dimensional case, we consider  sets
\begin{equation*}
\mis_{L}:=\big\{\mi\in\compseq:\sum_{j=1}^{\infty}\c_j\i_j\leq L\big\}
\end{equation*}
and the associated Smolyak algorithm
\begin{equation*}
\smol_{L}(\nm):=\sum_{\mis_{L}}\ddiff{\mix}\nm(\mi).
\end{equation*}
The following theorem is composed of results from \cite{GriebelOettershagen} on interpolation and integration of analytic functions; the calculations there transfer directly to the general setting. 
\begin{theorem}
	\label{thm:com:oettershagen}
	Let $L>0$ and define $N:=|\mis_{L}|=\work(\smol_{L}(\nm))$.
	\begin{enumerate}[(i)]
		
		\item Assume $\lc=0$.
		\begin{itemize}
			\item {\cite[Theorem 3.2]{GriebelOettershagen}} If there exists $\c_0>1$ such that $M:=M(\c_0,(\c_j)_{j=1}^{\infty}):=\sum_{j=1}^{\infty}\frac{1}{\exp(\c_j/\c_0)-1}<\infty$, then
			\begin{equation*}
			\norm{\smol_{L}(\nm)-\nm_{\infty}}{\vsc}\leq \frac{K_1 }{\c_0}\exp(\c_0M)N^{-(\c_0-1)},
			\end{equation*}
			which implies
			\begin{equation*}
			\work(\smol_{L}(\nm))\leq C(K,\beta_0,M)\epsilon^{-1/(\c_0-1)}
			\end{equation*}
			for $\epsilon:=\frac{K_1}{\c_0}\exp(\c_0M)N^{-(\c_0-1)}$.
			\item {\cite[Theorem 3.4]{GriebelOettershagen}} If $\c_j\geq \c_0j$ for $\c_0>0$, $j\geq 1$, then 
			\begin{equation*}
			\norm{\smol_{L}(\nm)-\nm_{\infty}}{\vsc}\leq \frac{2}{\c_0\sqrt{\log N}}N^{1+\frac{1}{4}\c_0-\frac{3}{8}\c_0(\log N)^{1/2}}.
			\end{equation*}
		\end{itemize}
		
		\item Assume $\lc>0$. 
		\begin{itemize}
			\item {\cite[Corollary 4.2 (i)]{GriebelOettershagen}} If there exist $\c_0>1$ and $\delta>0$ such that $M(\c_0,((1-\delta)\c_j)_{j=1}^{\infty})<\infty$, then
			\begin{equation*}
			\norm{\smol_{L}(\nm)-\nm_{\infty}}{\vsc}\leq C(K_1,\delta,\c_0,M,(\c_j)_{j\in\N},\lc)N^{-(\c_0-1)},
			\end{equation*}
			which implies
			\begin{equation*}
			\work(\smol_{L}(\nm))\leq C(K_1,\delta,\c_0,M,(\c_j)_{j\in\N},\lc)\epsilon^{-1/(\c_0-1)}
			\end{equation*}
			for $\epsilon:=C(K_1,\delta,\c_0,M,(\c_j)_{j\in\N},\lc)N^{-(\c_0-1)}$.
			\item {\cite[Corollary 4.2 (ii)]{GriebelOettershagen}} If $\c_j\geq \c_0j$ for $\c_0>0$, then for every $\hat{\c}_0<\c_0$ we have
			\begin{equation*}
			\norm{\smol_{L}(\nm)-\nm_{\infty}}{\vsc}\leq  \frac{C(\hat{\c}_0,M,b)}{\sqrt{\log N}}N^{1+\frac{\hat{\c}_0}{4}-\frac{3}{8}\hat{\c}_0(\log N)^{1/2}}.
			\end{equation*}
		\end{itemize}
	\end{enumerate}
	\hfill\qedsymbol
\end{theorem}

\begin{remark}
	For alternative approaches to infinite-dimensional problems, which allow even for exponential type work bounds, $\workf_j(\i_j)=K_{j,2}\exp(\w_j\i_j)$, consider for example \cite{dung2016hyperbolic,MR2719641,MR1881665}.
\end{remark}

%% file: content/applications.tex
\section{Applications}
\label{sec:com:applications}
\subsection{High-dimensional interpolation and integration}
\label{ssec:com:smolyak}
Smolyak introduced the algorithm that now bears his name in \cite{smolyak1963quadrature} to obtain efficient high-dimensional integration and interpolation formulas from univariate building blocks.
For example, assume we are given univariate interpolation formulas $S_{\i}$, $\i\in\N$ for functions in a Sobolev space $\vsunid$ that are based on evaluations in $2^{\i}$ points in $[0,1]$ and converge at the rate 
$$
\norm{S_{\i}-\Id}{\vsunid\to\vsunic}\leq C 2^{-k(\beta-\alpha)}
$$ for some $0\leq \alpha<\beta$. A straightforward high-dimensional interpolation formula is then the corresponding tensor product formula
\begin{equation*}
\bigotimes_{j=1}^{\cd}S_{\i_j}\colon \vsunid^{\otimes \cd}=:H^{\beta}_{\mix}([0,1]^{\cd})\to \vsunic^{\otimes \cd}=:H^{\alpha}_{\mix}([0,1]^{\cd})
\end{equation*}
for $(k_1,\dots,k_{\cd})\in\N^{\cd}$, where we consider both tensor product spaces to be completed with respect to the corresponding Hilbert space tensor norm \cite{Hackbusch2012}.
This can be interpreted as a numerical approximation method with values in a space of linear operators, 
\begin{equation*}
\nm(\mi):=\bigotimes_{j=1}^{\cd}S_{\i_j}\in\mathcal{L}(H^{\beta}_{\mix}([0,1]^{\cd}),H^{\alpha}_{\mix}([0,1]^{\cd}))=:\vsc,
\end{equation*}
whose discretization parameters $\mi=(\i_1,\dots,\i_{\cd})$ determine the resolution of interpolation nodes in each direction $j\in\{1,\dots,\cd\}$.

If we associate as work with $\nm(\mi)$ the number of required point evaluations, 
\begin{equation*}
\work(\nm(\mi)):=\prod_{j=1}^{\cd}2^{\i_j},  
\end{equation*}
then we are in the situation described in \Cref{ssec:com:finite}. Indeed, we have $\nm\in\vsbg(\vsc)$ with $\boundf_{j}(\i_j):=2^{-\i_j(\beta-\alpha)}$ by part (iii) of \Cref{exa}, since the operator norm of a tensor product operator between Hilbert space tensor products factorizes into the product of the operator norms of the constituent operators (see \cite[Proposition 4.127]{Hackbusch2012} and \cite[Section 26.7]{defant1992tensor}).

 In particular, the straightforward tensor product formulas $\nm(\i,\dots,\i)$ require the work
\begin{equation*}
\epsilon^{-\cd/(\beta-\alpha)}
\end{equation*}
to approximate the identity operator with accuracy $\epsilon>0$ in the operator norm, whereas Smolyak's algorithm $\smol_{L}(\nm)$ with an appropriate choice of $L=L(\epsilon)$ achieves the same accuracy with 
\begin{equation*}
\work(\smol_{L}(\nm))\lesssim \epsilon^{-1/(\beta-\alpha)}|\log\epsilon|^{(\cd-1)(1+1/(\beta-\alpha))},
\end{equation*}
according to \Cref{thm:main}. 
Here and in the following, we denote by $\lesssim$ estimates that hold up to factors that are independent of $\epsilon$.
As a linear combination of tensor product operators, Smolyak's algorithm $\smol_{L}(\nm)$ is a linear interpolation formula based on evaluations in the union of certain tensor grids. These unions are commonly known as \emph{sparse grids} \cite{Zenger91,bungartz2004sparse,garcke2012sparse}.
\begin{remark}
Interpolation of functions in general Banach spaces, with convergence measured in different general Banach spaces can be treated in the same manner. However, more care has to be taken with the tensor products. Once the algebraic tensor products of the function spaces are equipped with \emph{reasonable cross norms} \cite{Hackbusch2012} and completed, it has to be verified that the operator norm of linear operators between the tensor product spaces factorizes. Unlike for Hilbert spaces, this is not always true for general Banach spaces. However, it is true whenever the codomain is equipped with the \emph{injective tensor norm}, or when the domain is equipped with the \emph{projective tensor norm} \cite[Sections 4.2.9 and 4.2.12]{Hackbusch2012}. For example, the $L^{\infty}$-norm (and the similar $C^k$-norms) is an injective tensor norm on the product of $L^\infty$-spaces, while the $L^1$-norm is a projective norm on the tensor product of $L^1$-spaces.
\end{remark}

\subsection{Monte Carlo path simulation}
\label{ssec:com:MLMC}
Consider a stochastic differential equation (SDE)
\begin{equation}
\label{eq:com:stoch}
\begin{cases}
dS(t)=a(t,S(t))dt+b(t,S(t))dW(t)\quad0\leq t\leq T\\
S(0)=S_0\in \R^d,
\end{cases}
\end{equation}
with a Wiener process $W(t)$ and sufficiently regular coefficients $a,b\colon [0,T]\times \R^d\to\R$. A common goal in the numerical approximation of such SDE is to compute expectations of the form
\begin{equation*}
E[Q(S(T))],
\end{equation*} 
where $Q\colon\R^d\to\R$ is a Lipschitz-continuous quantity of interest of the final state $S(T)$. To approach this problem numerically, we first define random variables $S_{N}(t), 0\leq t\leq T$ as the forward Euler approximations of \Cref{eq:com:stoch} with $N\geq 1$ time steps. Next, we approximate the expectations $E[Q(S_{N}(T))]$ by Monte Carlo sampling using $M\geq 1$ independent samples $S^1_{N}(T),\dots,S^{M}_{N}(T)$ that are computed using independent realizations of the Wiener process. Together, this gives rise to the numerical approximation
\begin{equation*}
\nm(M,N):=\frac{1}{M}\sum_{i=1}^{M}Q(S^i_{N}(T)).
\end{equation*}
For fixed values of $M$ and $N$ this is a random variable that satisfies 
\begin{equation*}
\begin{split}
E[\left(\nm(M,N)-E[Q(S(T))]\right)^2]&=\left(E[\nm(M,N)]-E[Q(S(T))]\right)^2+\Var[\nm(M,N)]\\
&=(E[Q(S_N(T))]-E[Q(S(T))])^2+M^{-1}\Var[Q(S_{N}(T))]\\
& \lesssim N^{-2} + M^{-1},
\end{split}
\end{equation*}
where the last inequality holds by the weak rate of convergence of the Euler method \cite[Section 14.1]{Kloeden92} and by its $L^2$-boundedness as $N\to\infty$.
This shows that the random variables $\nm(M,N)$ converge to the limit $\nm_{\infty}=E[Q(S(T))]$, which itself is just a deterministic real number, in the sense of probabilistic mean square convergence as $M,N\to\infty$. To achieve a mean square error or order $\epsilon^2>0$, this straightforward approximation requires the simulation of $M\approx \epsilon^{-2}$ sample paths of \Cref{eq:com:stoch}, each with $N\approx \epsilon^{-1}$ time steps, which incurs the total work
\begin{equation*}
\work(\nm(M,N))=MN\approx\epsilon^{-3}.
\end{equation*}
Smolyak's algorithm allows us to achieve the same accuracy with the reduced work $\epsilon^{-2}$ of usual Monte Carlo integration. To apply the results of \Cref{ssec:com:finite}, we consider the reparametrized algorithm $\nm(k,l)$ with
$$
M_k:=M_0\exp(2k/3),
$$
$$
N_l:=N_0\exp(2l/3),
$$ for which the convergence and work parameters of \Cref{ssec:com:finite} attain the values $\c_j=1/3$, $\w_j=2/3$, and $\lc_j=\lw_j=0$, $j\in\{1,2\}$. (Here and in the following we implicitly round up non-integer values, which increases the required work only by a constant factor.) Indeed, we may write
\begin{equation*}
\nm(k,l)=\ml(\nm_1(k),\nm_2(l))), 
\end{equation*}
where $\nm_1(k)$, $k\in\N$ is the operator that maps random variables to an empirical average over $M_k$ independent samples, $\nm_2(l)$, $l\in\N$ is the random variable $Q(S_{N_l}(T))$, and $\ml$ denotes the application of linear operators to random variables. Since $\nm_1(k)$ converges in the operator norm to the expectation operator on the space of square integrable random variables at the usual Monte Carlo convergence rate $M_k^{-1/2}$ as $k\to\infty$, and $\nm_2(l)$ converges to $Q(S(T))$ at the strong convergence rate $N_l^{-1/2}$ of the Euler method in the $L^2$-norm \cite[Section 10.2]{Kloeden92} as $l\to\infty$, and since $\ml$ is linear in both arguments, the claimed values of the convergence parameters $\c_j$, $j\in\{1,2\}$ hold by part (iv) of \Cref{pro:com:2}.

\Cref{thm:main} now shows that choosing $L=L(\epsilon)$ such that
\begin{equation*}
E[(\smol_{L}(\nm)-E[Q(S(T))])^2]\leq \epsilon^2
\end{equation*}
incurs the work
\begin{equation}
\label{eq:com:MLwork}
\work(\smol_{L}(\nm))\lesssim \epsilon^{-2}|\log\epsilon|^{-3}.
\end{equation}

To link this result to the keyword \emph{multilevel approximation}, we observe that, thanks to our particular choice of parametrization, Smolyak's algorithm from \Cref{ssec:com:finite} takes the simple form
\begin{equation*}
\smol_{L}(\nm)=\sum_{k+l\leq L}\ddiff{\mix}\nm(k,l).
\end{equation*}
Since $\ddiff{\mix}=\ddiff{1}\circ\ddiff{2}$ and $\ddiff{1}=\dsum{1}^{-1}$ we may further write 
\begin{equation}
\label{eq:com:ml}
\begin{split}
\smol_{L}(\nm)=&\sum_{l=0}^{L}\sum_{k=0}^{L-l}\ddiff{\mix}\nm(k,l)\\
=&\sum_{l=0}^{L}\ddiff{2}\nm(L-l,l)\\
=&\frac{1}{M_L}\sum_{i=1}^{M_L}Q(S^{i}_{N_{0}}(T))+\sum_{l=1}^{L}\frac{1}{M_{L-l}}\sum_{i=1}^{M_{L-l}}\left(Q(S^{i}_{N_l}(T))-Q(S^{i}_{N_{l-1}}(T))\right),
\end{split}
\end{equation}
which reveals that Smolyak's algorithm employs a large number of samples from the coarse approximation $S_{N_0}(T)$, and subsequently improves on the resulting estimate of $E[Q(S(T))]$ by adding approximations of the expectations $E\left[Q(S_{N_{l}}(T))-Q(S_{N_{l-1}}(T))\right]$, $l\in\{1,\dots,L\}$ that are computed using less samples. 

\Cref{eq:com:ml} is a multilevel formula of the form analyzed in \cite{MR1629093} and \cite{giles2008multilevel}. Alternatively, this formula could also be deduced directly from the combination rule for triangles in \Cref{ssec:com:combination}. Compared to the analysis in \cite{giles2008multilevel}, our presentation has two shortcomings: First, our analysis only exploits the strong rate of the discretization method used to approximate \Cref{eq:com:stoch}. In the situation considered above, this does not affect the results, but for more slowly converging schemes a faster weak convergence rate may be exploited to obtain improved convergence rates. Second, the bound in  \Cref{eq:com:MLwork} is larger than that in \cite{giles2008multilevel} by the factor $|\log\epsilon|$. This factor can be removed by using independent samples for different values of $l$ in \Cref{eq:com:ml}, since we may then apply \Cref{rem:orth}.

\subsection{Multilevel quadrature}
\label{ssec:mlquadrature}

As in \Cref{exa} of \Cref{sec:com:truncation}, assume that we want to approximate the integral $\int_{[0,1]}\rs(x)\;dx\in\R$ using evaluations of approximations $\rs_{\ict}\colon[0,1]\to\R$, $\ict\in\N$. This is similar to the setting of the previous subsection, but with random sampling replaced by deterministic quadrature. 

 As before, denote by $S_{\i}$, $\i\in\N$ a sequence of quadrature formulas based on evaluations in $2^{\i}$ nodes. If we assume that point evaluations of $\rs_{\ict}$ require the work $\exp(\w\ict)$ for some $\w>0$, that
 \begin{equation*}
 \norm{\rs_{\ict}-\rs}{B}\lesssim 2^{-\kappa\ict}
 \end{equation*}
 for some $\kappa>0$ and a Banach space $B$ of functions on $[0,1]$ and that
 \begin{equation*}
 \norm{S_{\i}-\int_{[0,1]}\cdot\;dx}{B^*}\lesssim \exp(-\c\i)
 \end{equation*}
 for some $\c>0$, 
 then $\nm(\i,\ict):=S_{\i}\rs_{\ict}$ satisfies
 \begin{equation*}
 |S_{\i}\rs_{\ict}-\int_{[0,1]}\rs(x)\;dx|\lesssim \exp(-\c\i)+\exp(-\kappa\ict).
 \end{equation*}
 Hence, an accuracy of order $\epsilon>0$ can be achieved by setting 
 \begin{equation*}
\i:=-\log(\epsilon)/\c, \quad \ict:=-\log_2(\epsilon)/\kappa,
 \end{equation*}
 which requires the work
 \begin{equation*}
2^{\i}\exp(\w\ict)=\epsilon^{-1/\c-\w/\kappa}.
 \end{equation*}
 We have already shown the decay of the mixed differences,
 \begin{equation*}
 |\ddiff{\mix}\nm(\i,\ict)|\lesssim \exp(-\c \i )2^{-\kappa \ict},
 \end{equation*}
  in \Cref{exa}. Thus, 
 \Cref{thm:main} immediately shows that
 we can choose $L=L(\epsilon)$ such that Smolyak's algorithm satisfies 
 \begin{equation*}
 |\smol_{L}(\nm)-\int_{[0,1]}\rs(x)\;dx|\leq \epsilon,
 \end{equation*}
 with
 \begin{equation*}
 \work(\smol_{L}(\nm))\lesssim \epsilon^{-\max\{1/\c,\w/\kappa\}}|\log\epsilon|^{r}
 \end{equation*}
 for some $r=r(\c,\gamma,\kappa)\geq 0$. 
 
 As in \Cref{ssec:com:MLMC}, we may rewrite Smolyak's algorithm $\smol_{L}(\nm)$ in a multilevel form, which reveals that a Smolyak's algorithm employs a large number of evaluations of $\rs_{0}$, and subsequently improves on the resulting integral approximation by adding estimates of the integrals $\int_{[0,1]}\rs_{l}(x)-\rs_{l-1}(x)\;dx$, $l>0$,   that are computed using less quadrature nodes.

\subsection{Partial differential equations}
\label{ssec:com:combination}
The original Smolyak algorithm inspired two approaches to the numerical solution of partial differential equations (PDEs). The \emph{intrusive} approach is to solve discretizations of the PDE that are built on sparse grids. The \emph{non-intrusive} approach, which we describe here, instead applies the general Smolyak algorithm to product type discretizations whose resolution in the $j$-th direction is described by the parameter $\i_j$ \cite{griebel2014convergence,Zenger91}.

We discuss here how the non-intrusive approach can be analyzed using error expansions of finite difference approximations. For example, the work \cite{GriebelSchneiderZenger1992}, which introduced the name combination technique, exploited the fact that for the Poisson equation with sufficiently smooth data on $[0,1]^2$, finite difference approximations $u_{\i_1,\i_2}\in L^{\infty}([0,1]^2)$ with meshwidths $h_{j}=2^{-\i_{j}}$ in the directions $j\in\{1,2\}$ satisfy 
\begin{equation}
\label{eq:com:pdeexp}
u-u_{\i_1,\i_2}=w_1(h_1)+w_2(h_2)+w_{1,2}(h_1,h_2),
\end{equation}
where $u$ is the exact solution and $w_{1}(h_1),w_{2}(h_2),w_{1,2}(h_1,h_2)\in L^{\infty}([0,1]^2)$ are error terms that converge to zero in $L^{\infty}$ at the rates $\O(h_1^2)$, $\O(h_2^2)$, and $\O(h_1^2h_2^2)$, respectively.
Since the work required for the computation of $\nm(\i_1,\i_2):=u_{\i_1,\i_2}$ usually satisfies 
\begin{equation*}
\work(\nm(\i_1,\i_2))\approx (h_1h_2)^{-\gamma}
\end{equation*}
for some $\gamma\geq 1$ depending on the employed solver, an error bound of size $\epsilon>0$ could be achieved with the straightforward choice $\i_1:=\i_2:=-(\log_2\epsilon)/2$, which would require the work
\begin{equation*}
\work(\nm(\i_1,\i_2))\approx \epsilon^{-\gamma}.
\end{equation*}

 Since \Cref{eq:com:pdeexp} in combination with part (iii) of \Cref{pro:com:2} shows that $\nm\in \vsb_{(\boundf_j)_{j=1}^{2}}$ with $\boundf_j(\i):=2^{-2\i_j}$, we may deduce from \Cref{thm:main} that Smolyak's algorithm applied to $\nm$ requires only the work 
 \begin{equation*}
 \epsilon^{-\gamma/2}|\log\epsilon|^{1+\gamma/2}
 \end{equation*}
 to achieve the same accuracy.
The advantage of Smolyak's algorithm becomes even more significant in higher dimensions. All that is required to generalize the analysis presented here to high-dimensional problems, as well as to different PDE and different discretization methods, are error expansions such as \Cref{eq:com:pdeexp}.

\subsection{Uncertainty quantification}
A common goal in uncertainty quantification \cite{BabuskaTemponeZouraris2004,LeMaitreKnio2010,haji2015multi} is the approximation of response surfaces
\begin{equation*}
\domPS\ni \psmi\mapsto \rs(\psmi):=Q(\pde_{\psmi})\in\R.
\end{equation*}
Here, $\psmi\in\domPS\subset\R^m$ represents parameters in a PDE and $Q(\pde_{\psmi})$ is a real-valued quantity of interest of the corresponding solution $\pde_{\psmi}$.  For example, a thoroughly studied problem is the parametric linear elliptic second order equation with coefficients $a\colon U\times \domPS\to\R$, 
\begin{equation*}
\begin{cases}
-\nabla_x \cdot(a(x,\psmi)\nabla_x \pde_{\psmi}(x))=g(x)\quad&\text{ in }U\subset\R^d\\
\quad\quad\quad\pde_{\psmi}(x)=0  \quad&\text{ on }\partial U,
\end{cases}
\end{equation*}
whose solution for any fixed $\psmi\in\domPS$ is a function $\pde_{\psmi}\colon U\to \R$. 

Approximations of response surfaces may be used for optimization, for worst-case analysis, or to compute statistical quantities such as mean and variance in the case where $\domPS$ is equipped with a probability distribution. The non-intrusive approach to compute such approximations, which is known as \emph{stochastic collocation} in the case where $\domPS$ is equipped with a probability distribution, is to compute the values of $\rs$ for finitely many values of $\psmi$ and then interpolate. For example, if we assume for simplicity that $\domPS=\prod_{j=1}^{m}[0,1]$, then we may use, as in \Cref{ssec:com:smolyak}, a sequence of interpolation operators $S_{\i}\colon H^{\beta}([0,1])\to H^{\alpha}([0,1])$ based on evaluations in $(\ps_{\i,i})_{i=1}^{2^\i}\subset [0,1]$.
However, unlike in \Cref{ssec:com:smolyak}, we cannot compute values of $\rs$ exactly but have to rely on a numerical PDE solver. If we assume that this solver has discretization parameters $\mict=(\ict_{1},\dots,\ict_{d})\in \N^{d}$ and returns approximations $\pde_{\psmi,\mict}$ such that the functions
\begin{equation*}
\begin{split}
\rs_{\mict}\colon \domPS&\to\R\\
\psmi&\mapsto \rs_{\mict}(\psmi):=Q(\pde_{\psmi,\mict})
\end{split}
\end{equation*}
are elements of $H^{\beta}_{\mix}([0,1]^m)$, then we may define the numerical approximation method
\begin{equation*}
\begin{split}
&\nm\colon \N^{m}\times\N^{d}\to H^{\alpha}_{\mix}([0,1]^m)=:\vsc\\
&\nm(\mi,\mict):=\Big(\bigotimes_{j=1}^{m}S_{\i_j}\Big)\rs_{\mict},
\end{split}
\end{equation*}
with $n:=m+d$ discretization parameters.

At this point the reader should already be convinced that straightforward approximation is a bad idea. We therefore omit this part of the analysis, and directly move on to the application of Smolyak's algorithm. To do so, we need to identify functions $\boundf_{j}\colon \N\to\Rp$ such that $\nm\in \vsb_{(\boundf_j)_{j=1}^{n}}(\vsc)$. For this purpose, we write $\nm$ as
\begin{equation*}
\nm(\mi,\mict)=\ml(\nm_1(\mi),\nm_2(\mict)),
\end{equation*}
where 
\begin{equation*}
\nm_1(\mi):=\bigotimes_{j=1}^{m}S_{\i_j}\in \lin\Big(H^{\beta}_{\mix}([0,1]^m);H^{\alpha}_{\mix}([0,1]^m)\Big)=:\vsc_1 \quad \forall \mi\in \N^m
\end{equation*}
\begin{equation*}
\nm_2(\mict):=\rs_{\mict}\in H^{\beta}_{\mix}([0,1]^m)=:\vsc_2\quad\forall \mict \in \N^d
\end{equation*}
and
\begin{equation*}
\ml\colon Y_1 \times  Y_2 \to \vsc
\end{equation*} 
is the application of linear operators in $\vsc_1$ to functions in $\vsc_2$. Since $\ml$ is continuous and multilinear, we may apply part (iv) of \Cref{pro:com:2} to reduce our task to the study of $\nm_1$ and $\nm_2$. The first part can be done exactly as in \Cref{ssec:com:smolyak}. The second part can be done similarly to \Cref{ssec:com:combination}. However, we now have to verify not only that the approximations $\pde_{\psmi,\mict}$ converge to the exact solutions $\pde_{\psmi}$ for each fixed value of $\psmi$ as $\min_{j=1}^{d}\ict_{j}\to\infty$, but that this convergence holds in some uniform sense over the parameter space.

More specifically, let us denote by $\ddiff{\mix}^{(\mict)}$ the mixed difference operator with respect to the parameters $\mict$ and let us assume that 
\begin{equation*} \norm{\ddiff{\mix}^{(\mict)}\rs_{\mict}}{H^{\beta}_{\mix}([0,1]^m)}\lesssim \prod_{j=1}^{d}\exp(-\kappa_j \ict_{j})=:\prod_{j=1}^{d}\boundf^{(2)}_{j}(\ict_j)\quad\forall \mict\in\N^{d}.
\end{equation*}
For example, such bounds are proven in \cite{harbrecht2013multilevel,haji2015multi}. 
If the interpolation operators satisfy as before
\begin{equation*}
\norm{S_{\i}-\Id}{H^{\beta}([0,1])\to H^{\alpha}([0,1])}\lesssim 2^{-\i(\beta-\alpha)}=:\boundf^{(1)}(\i)\quad\forall\,\i\in\N,
\end{equation*} then the results of \Cref{ssec:com:smolyak} together with part (iv) of \Cref{pro:com:2} shows that 
\begin{equation*}
\nm\in \vsb_{(\boundf^{(1)})_{j=1}^{m}\cup (\boundf^{(2)}_j)_{j=1}^{d}}(\vsc).
\end{equation*}
If we further assume that the work required by the PDE solver with discretization parameters $\mict$ is bounded by $\exp(\wmi^{(2)}\cdot\mict)$ for some $\wmi\in\Rp^{d}$, then we may assign as total work to the algorithm $\nm(\mi,\mict)$ the value
\begin{equation*}
\work(\nm(\mi,\mict)):=2^{|\mi|_{1}}\exp(\wmi\cdot \mict),
\end{equation*} which is the number of required samples, $2^{|\mi|_{1}}$, times the bound on the work per sample, $\exp(\wmi\cdot\mict)$. Thus, by \Cref{thm:main}, Smolyak's algorithm achieves the accuracy
\begin{equation*}
\norm{\smol_{L}(\nm)-\rs}{\vsc}\lesssim \epsilon
\end{equation*}
with
\begin{equation*}
\work(\smol_{L}(\nm))\lesssim \epsilon^{-\rho}|\log\epsilon|^{r},
\end{equation*}
where $\rho:=\max\{1/(\beta-\alpha),\max\{\w_j/\kappa_j\}_{j=1}^{d}\}$ and $r\geq 0$ as in \Cref{ssec:com:finite}.

\section{Conclusion}
We showed how various existing efficient numerical methods for integration, Monte Carlo simulations, interpolation, the solution of partial differential equations, and uncertainty quantification can be derived from two common underlying principles: decomposition and efficient truncation. The analysis of these methods was divided into proving decay of mixed differences by means of \Cref{pro:com:2} and then applying general bounds on exponential sums in form of \Cref{thm:main}. 

Besides simplifying and streamlining the analysis of existing methods, we hope that the framework provided in this work encourages novel applications. Finally, we believe that the general version of Smolyak's algorithm presented here may be helpful in designing flexible and reusable software implementations that can be applied to future problems without modification.

%% file: content/exponentialsums.tex
\begin{lemma}
Let $\w_j>0$, $\c_j>0$, and $\lw_j>0$ for $j\in\{1,\dots,\cd\}$. Then
\begin{equation*}
\sum_{(\cmi+\wmi)\cdot\mi\leq L}\exp(\wmi\cdot\mi)(\mi+\mathbf{1})^{\lwmi}\leq C(\wmi,\lwmi,\cd)\exp(\mu L)(L+1)^{\cd^*-1+\lw^*},
\end{equation*}
where $\rho:=\max_{j=1}^{\cd}\w_j/\c_j$, $\mu:=\frac{\rho}{1+\rho}$, $J:=\{j\in\{1,\dots,\cd\}:\w_j/\c_j=\rho\}$, $\cd^*:=|J|$, $\lw^*:=\sum_{j\in J}\lw_j$, and $(\mi+\mathbf{1})^{\lwmi}:=\prod_{j=1}^{\cd}(\i_j+1)^{\lw_j}$.
\end{lemma}
\begin{proof}
First, we assume without loss of generality that the dimensions are ordered according to whether they belong to $J$ or $\comp{J}:=\{1,\dots,\cd\}\setminus J$. To avoid cluttered notation we then separate dimensions by plus or minus signs in the subscripts; for example, we write 
$\lwmi=(\lwmi_J,\lwmi_{\comp{J}})=:(\lwmi_+,\lwmi_-)$.

Next, we may replace the sum by an integral over $\{(\cmi+\wmi)\cdot\xmi\leq L\}$. Indeed, by monotonicity we may do so if we replace $L$ by $L+|\cmi+\wmi|_1$, but looking at the final result we observe that a shift of $L$ only affects the constant $C(\wmi,\lwmi,\cd)$. 

Finally, using a change of variables $y_{j}:=(\c_{j}+\w_{j})x_{j}$ and the shorthand $\mumi:=\wmi/(\cmi+\wmi)$ (with componentwise division) we obtain
\begin{equation*}
\begin{split}
\int_{(\cmi+\wmi)\cdot \xmi\leq L}\exp(\wmi\cdot \xmi)&(\xmi+\mathbf{1})^{\lwmi}\;d\xmi\leq C\int_{|\ymi|_1\leq L}\exp(\mumi\cdot \ymi)(\ymi+\mathbf{1})^{\lwmi}\;d\ymi\\
=C\int_{|\ymi_{+}|_1\leq L}&\exp(\mumi_{+}\cdot \ymi_{+})(\ymi_{+}+\mathbf{1})^{\lwmi_{+}}\int_{|\ymi_{-}|_{1}\leq L-|\ymi_{+}|_1}\exp(\mumi_{-}\cdot\ymi_{-})(\ymi_{-}+\mathbf{1})^{\lwmi_{-}}\;d\ymi_{-}\;d\ymi_{+}\\
\leq C\int_{|\ymi_{+}|_1\leq L}&\exp(\mu|\ymi_{+}|_1)(\ymi_{+}+\mathbf{1})^{\lwmi_{+}}\int_{|\ymi_{-}|_{1}\leq L-|\ymi_{+}|_1}\exp(\mu_{-}|\ymi_-|_{1})(\ymi_{-}+\mathbf{1})^{\lwmi_{-}}\;d\ymi_{-}\;d\ymi_{+}=(\star),
\end{split}
\end{equation*}
where the last equality holds by definition of $\mu=\max\{\mumi_{+}\}$ and $\mu_{-}:=\max\{\mumi_{-}\}$. We use the letter $C$ here and in the following to denote quantities that depend only on $\wmi,\lwmi$ and $\cd$ but may change value from line to line. Using $(\ymi_{
+}+\mathbf{1})^{\lwmi_{+}}\leq (|\ymi_{+}|_1+1)^{|\lwmi_{+}|_1}$ and $(\ymi_{-}+\mathbf{1})^{\lwmi_{-}}\leq (|\ymi_{-}|_1+1)^{|\lwmi_{-}|_1}$ and the linear change of variables $\ymi\mapsto (|\ymi|_{1},y_2,\dots,y_n)$ in both integrals, we obtain 
\begin{equation*}
\begin{split}
(\star)&\leq C\int_{|\ymi_{+}|_1\leq L}\exp(\mu|\ymi_{+}|_1)(|\ymi_{+}|_{1}+1)^{|\lwmi_{+}|_{1}}\int_{|\ymi_{-}|_{1}\leq L-|\ymi_{+}|_1}\exp(\mu_{-}|\ymi_-|_{1})(|\ymi_{-}|_1+1)^{|\lwmi_{-}|_1}\;d\ymi_{-}\;d\ymi_{+}\\
&\leq C\int_{0}^L\exp(\mu u)(u+1)^{|\lwmi_{+}|_{1}}u^{|J|-1}\int_{0}^{L-u}\exp(\mu_{-}v)(v+1)^{|\lwmi_{-}|_1}v^{|\comp{J}|-1}\;dv\;du\\
&\leq C(L+1)^{|\lwmi_{+}|_{1}}L^{|J|-1}\int_{0}^L\exp(\mu u)((L-u)+1)^{|\lwmi_{-}|_1}(L-u)^{|\comp{J}|-1}\int_{0}^{L-u}\exp(\mu_{-}v)\;dv\;du\\
&\leq C(L+1)^{|\lwmi_{+}|_{1}+|J|-1}\int_{0}^L\exp(\mu u)(L-u+1)^{|\lwmi_{-}|_1}(L-u)^{|\comp{J}|-1}\exp(\mu_-(L-u))\;du\\
&=C(L+1)^{|\lwmi_{+}|_{1}+|J|-1}\exp(\mu L)\int_{0}^L\exp(-(\mu-\mu_{-})w)(w+1)^{|\lwmi_{-}|_1}w^{|\comp{J}|-1}\;dw\\
&\leq C(L+1)^{|\lwmi_{+}|_{1}+|J|-1}\exp(\mu L),
 \end{split}
\end{equation*}
where we used supremum bounds for both integrals for the third inequality, the change of variables $w:=L-u$ for the penultimate equality, and the fact that $\mu>\mu_{-}$ for the last inequality.
\end{proof}

\begin{lemma}
Let $\w_j>0$, $\c_j>0$, and $\lc_j>0$ for $j\in\{1,\dots,\cd\}$. Then
\begin{equation*}
\sum_{(\cmi+\wmi)\cdot\mi>L}\exp(-\cmi\cdot\mi)(\mi+\mathbf{1})^{\lcmi} \leq C(\cmi,\lcmi,\cd)\exp(-\nu L)(L+1)^{\cd^{*}-1+\lc^*},
\end{equation*}
where $\rho:=\max_{j=1}^{\cd}\w_j/\c_j$, $\nu:=\frac{1}{1+\rho}$, $J:=\{j\in\{1,\dots,\cd\}:\w_j/\c_j=\rho\}$, $\cd^*:=|J|$, $\lc^*:=\sum_{j\in J}\lw_j$, and $(\mi+\mathbf{1})^{\lcmi}:=\prod_{j=1}^{\cd}(\i_j+1)^{\lc_j}$.
\end{lemma}

\begin{proof}
First, we assume without loss of generality that the dimensions are ordered according to whether they belong to $J$ or $\comp{J}$. To avoid cluttered notation we then separate dimensions by plus or minus signs in the subscripts; for example, we write 
$\lcmi=(\lcmi_J,\lcmi_{\comp{J}})=:(\lcmi_+,\lcmi_-)$.

Next, we may replace the sum by an integral over $\{(\cmi+\wmi)\cdot\xmi>L\}$. Indeed, by monotonicity we may do so if we replace $L$ by $L-|\cmi+\wmi|_1$, but looking at the final result we observe that a shift of $L$ only affects the constant $C(\cmi,\lcmi,\cd)$. 

Finally, using a change of variables $y_{j}:=(\c_{j}+\w_{j})x_{j}$ and the shorthand $\numi:=\cmi/(\cmi+\wmi)$ (with componentwise division) we obtain
\begin{equation*}
\begin{split}
&\int_{(\cmi+\wmi)\cdot \xmi>L}\exp(-\cmi\cdot \xmi)(\xmi+\mathbf{1})^{\lcmi}\;d\xmi\leq C\int_{|\ymi|_1>L}\exp(-\numi\cdot \ymi)(\ymi+\mathbf{1})^{\lcmi}\;d\ymi\\
&=C\int_{|\ymi_{+}|_1> L}\exp(-\numi_{+}\cdot \ymi_{+})(\ymi_{+}+\mathbf{1})^{\lcmi_{+}}\int_{|\ymi_{-}|_{1}>(L-|\ymi_{+}|_1)^{+}}\exp(-\numi_{-}\cdot\ymi_{-})(\ymi_{-}+\mathbf{1})^{\lcmi_{-}}d\ymi_{-}d\ymi_{+}\\
&\leq C\int_{|\ymi_{+}|_1> L}\exp(-\nu|\ymi_{+}|_1)(\ymi_{+}+\mathbf{1})^{\lcmi_{+}}\int_{|\ymi_{-}|_{1}>(L-|\ymi_{+}|_1)^{+}}\exp(-\nu_{-}|\ymi_-|_{1})(\ymi_{-}+\mathbf{1})^{\lcmi_{-}}d\ymi_{-}d\ymi_{+}\\
&=:(\star),
\end{split}
\end{equation*}
where the last equality holds by definition of $\nu=\max\{\numi_{+}\}$ and $\nu_{-}:=\max\{\numi_{-}\}$. We use the letter $C$ here and in the following to denote quantities that depend only on $\cmi,\lcmi$ and $\cd$ but may change value from line to line. Using $(\ymi_{
+}+\mathbf{1})^{\lcmi_{+}}\leq (|\ymi_{+}|_1+1)^{|\lcmi_{+}|_1}$ and $(\ymi_{-}+\mathbf{1})^{\lcmi_{-}}\leq (|\ymi_{-}|_1+1)^{|\lcmi_{-}|_1}$ and the linear change of variables $\ymi\mapsto (|\ymi|_{1},y_2,\dots,y_n)$ in both integrals, we obtain  
\begin{equation*}
\begin{split}
(\star)\leq&  C\int_{|\ymi_{+}|_1>0}\exp(-\nu|\ymi_{+}|_1)(|\ymi_{+}|_{1}+1)^{|\lcmi_{+}|_{1}}\int_{|\ymi_{-}|_{1}>(L-|\ymi_{+}|_1)^{+}}\exp(-\nu_{-}|\ymi_-|_{1})(|\ymi_{-}|_1+1)^{|\lcmi_{-}|_1}d\ymi_{-}d\ymi_{+}\\
\leq& C\int_{0}^\infty \exp(-\nu u)(u+1)^{|\lcmi_{+}|_{1}}u^{|J|-1}\int_{(L-u)^{+}}^{\infty}\exp(-\nu_{-}v)(v+1)^{|\lcmi_{-}|_1}v^{|\comp{J}|-1}\;dv\;du\\
=& C\int_0^{L}\exp(-\nu u)(u+1)^{|\lcmi_{+}|_{1}+|J|-1}\int_{L-u}^{\infty}\exp(-\nu_{-}v)(v+1)^{|\lcmi_{-}|_1+|\comp{J}|-1}\;dv\;du\\
& + C\int_L^{\infty}\exp(-\nu u)(u+1)^{|\lcmi_{+}|_{1}+|J|-1}\int_{0}^{\infty}\exp(-\nu_{-}v)(v+1)^{|\lcmi_{-}|_1+|\comp{J}|-1}\;dv\;du\\
=&:(\star\star)+(\star\star\star).
 \end{split}
\end{equation*}
To bound $(\star\star)$, we estimate the inner integral using the inequality $\int_{a}^{\infty}\exp(-b v)(v+1)^{c}\,dv\leq C\exp(-b a)(a+1)^c$ \cite[(8.11.2)]{NIST:DLMF}, which is valid for all positive $a, b, c$:
\begin{equation*}
\begin{split}
(\star\star)&\leq C\int_0^{L}\exp(-\nu u)(u+1)^{|\lcmi_{+}|_{1}+|J|-1}\exp(-\nu_{-}(L-u))(L-u+1)^{|\lcmi_{-}|_1+|\comp{J}|-1}\;du\\
&\leq C(L+1)^{|\lcmi_{+}|_{1}+|J|-1}\int_0^{L}\exp(-\nu(L-w))\exp(-\nu_{-}w)(w+1)^{|\lcmi_{-}|_1+|\comp{J}|-1}\;dw\\
&=C(L+1)^{|\lcmi_{+}|_{1}+|J|-1}\exp(-\nu L)\int_0^{L}\exp(-(\nu_{-}-\nu)w)(w+1)^{|\lcmi_{-}|_1+|\comp{J}|-1}\;dw\\
&\leq C(L+1)^{|\lcmi_{+}|_{1}+|J|-1}\exp(-\nu L),
\end{split}
\end{equation*}
where we used a supremum bound and the change of variables $w:=L-u$ for the second inequality, and the fact that $\nu_{-}>\nu$ for the last inequality. Finally, to bound $(\star\star\star)$, we observe that the inner integral is independent of $L$, and bound the outer integral in the same way we previously bounded the inner integral. This shows
\begin{equation*}
(\star\star\star)\leq C\exp(-\nu L)(L+1)^{|\lcmi_{+}|_{1}+|J|-1}.
\end{equation*}
\end{proof}